\documentclass[12pt, twoside, leqno]{article}

\usepackage{amsmath,amsthm}
\usepackage{amssymb}

\usepackage{enumitem}

\usepackage{tikz-cd}

\newtheorem{theorem}{Theorem}
\newtheorem{fact}{Fact}
\newtheorem*{lemma}{Lemma}

\newcommand{\T}{\mathbb{T}}
\newcommand{\TT}{\mathcal{T}}

\newcommand{\HH}{\mathcal{H}}
\newcommand{\R}{\mathbb{R}}
\newcommand{\nul}{\mathrm{null}}

\DeclareRobustCommand{\divby}{%
	\mathrel{\text{\vbox{\baselineskip.65ex\lineskiplimit0pt\hbox{.}\hbox{.}\hbox{.}}}}%
}

\begin{document}
	
	\baselineskip=17pt
	
	\title{Absence of local unconditional structure in spaces of smooth functions on the torus of arbitrary dimension}
	
	\author{Anton Tselishchev\\
	Chebyshev Laboratory\\
	St. Petersburg State University\\
	14th Line V.O., 29\\
	Saint Petersburg 199178, Russia\\
	E-mail: celis-anton@yandex.ru}

	\date{}
	
	\maketitle
	
	%%%%%%%
	
	\renewcommand{\thefootnote}{}
	
	
	
	\renewcommand{\thefootnote}{\arabic{footnote}}
	\setcounter{footnote}{0}
	
	%%%%%%%
	
	\begin{abstract}
		Consider a finite collection $\{T_1, \ldots, T_J\}$ of differential operators with constant coefficients on $\mathbb{T}^n$ ($n\geq 2$) and the space of smooth functions generated by this collection, namely, the space of functions $f$ such that $T_j f \in C(\mathbb{T}^n)$, $1\leq j\leq J$. We prove that if there are at least two linearly independent operators among their senior parts (relative to some mixed pattern of homogeneity), then this space does not have local unconditional structure. This fact generalizes the previously known result that such spaces are not isomorphic to a complemented subspace of $C(S)$.
	\end{abstract}
	
	\section{Introduction}
	
	It is a well-known fact of Banach space theory that for $n\geq 2$ the space $C^k(\T^n)$ of $k$ times continuously differentiable functions on the torus $\T^n$ is not isomorphic to the space $C(\T^n)$. This fact was first announced in $\cite{Grot}$ and later generalized in many directions (see \cite{Henk, KislFactor, Kisl0, KwaPel, Sid, PelSen, KisSid, KisMaks, KisMaks2, Maks}). However, the most general and natural framework was introduced only in the quite recent paper \cite{KMSpap} (see also the preprint \cite{KMSprepr} for the two-dimensional case).
	
	We proceed to the description of this framework. For a multiindex $\alpha=(\alpha_1, \ldots, \alpha_n)\in \mathbb{Z}_+^n$ ($\mathbb{Z}_+$ stands for the set of nonnegative integers) we denote by $D^\alpha$ the differential monomial $\partial_1^{\alpha_1}\ldots \partial_n^{\alpha_n}$ (note that we have $\partial z^l=2\pi i l z^l$, $z\in\mathbb{T}$). We fix some collection $\TT=\{T_1, T_2, \ldots, T_J\}$ of differential operators with constant coefficients on the torus $\T^n$. This means that each of these operators is a linear combination of the operators $D^\alpha$ with different multiindices $\alpha$. The number $\alpha_1+\ldots+\alpha_n$ is called the order of $D^\alpha$ and the order of the operator $T_j$ is the highest order among all differential monomials involved in it. 
	
	We consider the following seminorm on the trigonometric polynomials $f$:
	$$
	\|f\|_{\TT} =\max_{1\leq j\leq J} \|T_j f\|_{C(\T^n)}.
	$$ 
	The space which is generated by this seminorm (i.e., by factorizing over the null space and completion) is denoted by $C^\TT(\T^n)$. For example, if $\TT$ is the set of all differential monomials of order at most $k$, then this space coincides with the space $C^k(\T^n)$.
	
	The space $C^\TT(\T^n)$ is identified isometrically with a closed subspace of the space $\ell^\infty_J(C(\T^n))$ via the following map:
	$$
	f\mapsto (T_j f)_{j\le J}.
	$$
	The space $W_p^\TT(\T^n)$ for $p>0$ is defined similarly and is identified isometrically with a closed subspace of $\ell_J^\infty (L^p (\mathbb{T}^n))$ via the same map. Such spaces will appear in the paper for $p\in\{1/2, 1, 2\}$.
	
	The spaces $C^\TT(\T^n)$ were already studied in the papers \cite{KisMaks, KisMaks2}. In particular, the following statement was proven. Suppose that the order of every operator $T_j$ does not exceed $k$. Let us drop all the monomials in each $T_j$ with orders strictly smaller than $k$. If after that there are at least two linearly independent operators among the remaining senior parts of $T_j$, then the space $C^\TT(\T^n)$ is not isomorphic to a complemented subspace of $C(S)$. (We denote by $S$ an arbitrary uncountable compact metric space. According to the Milutin theorem, all the resulting $C(S)$ spaces are isomorphic.) However, if all senior parts are multiples of one of them, the situation was unclear.
	
	After that, in the paper \cite{KMSpap} a refinement of this statement was proven. In order to state it, we need the concept of mixed homogeneity.
	
	Fix some \textit{pattern of mixed homogeneity}, that is, a hyperplane $\Lambda$ in $\mathbb{R}^n$ intersecting the positive semiaxes. The equation of such a hyperplane is of the form $\langle x,\rho \rangle=1$ for some vector $\rho$ with positive coordinates. We call such a hyperplane admissible if all points $\alpha$ such that $D^\alpha$ is involved in some $T_j$ lie below $\Lambda$ or on it. This means that the following inequality holds:
	$$
	\sum_{j=1}^N \rho_j\alpha_j\leq 1.
	$$
	
	Now we define the senior part of $T_j$ as the linear combination of all differential monomials involved in $T_j$ whose multiindices lie on the hyperplane $\Lambda$ and the junior part as the remaining part of $T_j$. The senior part is denoted by $\sigma_j$ and the junior by $\tau_j$. In \cite{KMSpap} it is proven that if there are at least two linearly independent operators among these senior parts, then the bidual of the space $C^\TT(\T^n)$ is not isomorphic to a complemented subspace of a $C(S)$ space. 
	
	The same statement, but for the 2-dimensional torus $\T^2$, was also proved in the preprint \cite{KMSprepr}. However, it turned out that it is at least not clear how to derive the statement for an arbitrary dimension from it (see \cite{KMSpap} for details). Thus, in the proof of the theorem in the general case some significant technical difficulties appeared.
	
	Besides that, we note that for the case when all operators in $\TT$ are monomials, the statement about non isomorphism to a complemented subspace of a $C(S)$ space is not the most general. Let us give some necessary definitions. A Banach space $X$ is said to have local unconditional structure if there exists a constant $C>0$ such that for any finite-dimensional subspace $F\subset X$ there exists a Banach space $E$ with $1$-unconditional basis and two linear operators $R: F \rightarrow E$ and $S: E\rightarrow X$ such that $SRx = x$ for all $x\in F$ and $\|S\|\cdot \|R\| \leq C$. A basis $\{e_n\}$ is $1$-unconditional if for any numbers $\varepsilon_n$ with $|\varepsilon_n| \leq 1$ and any finite sequence $(\alpha_n)$ we have $\|\sum \varepsilon_n  \alpha_n x_n\| \leq \|\sum \alpha_n x_n\|$. We note that in general the property of having local unconditional structure is not stable when passing to closed subspaces (however it is true for complemented subspaces).
	
	In the paper \cite{KisSid} the following statements were proved in the case when all $T_j$ are differential monomials. If, again, there are at least two linearly independent operators among the senior parts of $T_j$ (here, each $T_j$ is a differential monomial, so these senior parts are equal either to $T_j$ or to zero), then the space $C^\TT(\T^n)$ does not have local unconditional structure. Besides that, if the space $C^\TT(\T^n)^*$ is isomorphic to a subspace of a space $Y$ with local unconditional structure, then $Y$ contains the spaces $\ell_\infty^k$ uniformly. This means that there exist subspaces $Y_k$ of $Y$ such that $\dim Y_k=k$, and invertible operators $T_k:Y_k\rightarrow \ell_\infty^k$ with the property $\|T_k\| \cdot \|T_k^{-1}\|\leq C$. Since the space $C(S)^*$ has local unconditional structure but does not contain the spaces $\ell_\infty^k$ uniformly, these statements imply that $C^\TT(\T^n)$ is not isomorphic to a quotient of a $C(S)$ space.
	
	Our goal in this paper is to prove the same statements but for the general operators $T_j$. Now, we state the main result.
	
	\begin{theorem}
		Suppose that for the collection $\TT$ there are at least two linearly independent operators among the $\sigma_j$ \emph(for some choice of an admissible hyperplane $\Lambda$\emph). Then $C^{\TT}(\T^n)$ does not have local unconditional structure. Furthermore, if $C^\TT(\T^n)^*$ is isomorphic to a subspace  of a space $Y$ with local unconditional structure, then $Y$ contains the spaces $\ell_\infty^k$ uniformly.
	\end{theorem}
	
	The statement about the absence of local unconditional structure was also proved in the preprint \cite{Tsel} for the two-dimensional case. However, the situation here is the same as in the preprint \cite{KMSprepr} and the paper \cite{KMSpap} --- the proof for arbitrary dimensions is significantly more difficult technically.
	
	The main ingredient of the proof of the theorem about non isomorphism to a complemented subspace of a $C(S)$ space from \cite{KMSpap} is the new embedding theorem which was also proved in \cite{KMSpap}. We are going to use it, so we state it here.
	
	\begin{fact}
		Let $\rho_0, \ldots, \rho_l$ be finite complex measures on $\mathbb{R}^n$ \emph($n\geq 2$\emph) and consider the following functions on $\mathbb{R}^2$: $h_j(\xi_1, \xi_2)=\hat{\rho}_j(\xi_1, |\xi_2|^{\theta_2}, \ldots, |\xi_2|^{\theta_n})$, where $\theta_2, \ldots, \theta_n$ are positive numbers. Fix some real number $\varkappa\geq 1$ and suppose that the functions $\eta_j$, $j=1,\ldots, l$, on $\mathbb{R}^2$ satisfy the following system of equations:
		\begin{gather*}
		\xi_1\eta_1(\xi)=h_0(\xi); \quad  \xi_1\eta_j(\xi)-|\xi_2|^\varkappa \eta_{j-1}(\xi) = h_{j-1}(\xi), \ \, j=2,\ldots, l; \\ -|\xi_2|^\varkappa \eta_l(\xi) = h_l(\xi).
		\end{gather*}
		
		Then the following inequality holds:
		$$
		\max_j \bigg(\int_{\mathbb{R}^2}|\eta_j(\xi)|^2 |\xi_2|^{\varkappa-1} d\xi\bigg)^{1/2} \lesssim \max_j\|\rho_j\|.
		$$
	\end{fact}
	
	By $\|\rho_j\|$ we mean the total variation of the measure $\rho_j$. We use the symbol ``$\lesssim$" to indicate that the left hand side does not exceed some positive constant times the right hand side. Also, the notation $A\asymp B$ means that $A\lesssim B$ and $B\lesssim A$.
	
	It is worth noting that the system of equations in Fact 1 is solvable iff the following compatibility condition holds (see \cite{KMSpap}): $\sum_{j=0}^l \xi_1^j\xi_2^{-j\varkappa}h_j(\xi)=0$ for $\xi\neq 0$.
	
	Mainly by combining the techniques from \cite{KisSid} and \cite{KMSpap}, we are going to prove Theorem 1, which seems to be the most general in this context.
	
	\section{Some technical simplifications}
	
	\subsection{Rotation of the hyperplane}
	
	First of all, we note that we can do the same procedures as in \cite{KMSpap} in order to rotate the hyperplane $\Lambda$ and to modify the collection $\TT$. We describe these procedures without going into details. All the proofs of the purely technical statements may be found in \cite{KMSpap} and we refer to them in the text.
	
	A distribution $F$ on $\T^n$ is called proper if $\hat{F}(m)=0$ whenever $m_j=0$ for some $j$. For example $z_1^{m_1}\ldots z_n^{m_n}$ is proper if $m_j\neq 0$ for every $j$. The space $C_0^\TT(\T^n)$ is the subspace of $C^\TT(\T^n)$ consisting of proper functions. It is a complemented subspace of $C^\TT (\T^n)$ and therefore it suffices to prove Theorem 1 with $C_0^\TT(\T^n)$ instead of $C^\TT(\T^n)$.
	
	Next, we ``slightly rotate" the hyperplane $\Lambda$ in such a way that its intersection with the set of multiindices corresponding to the differential monomials involved in the operators in $\TT$ does not change, it is still admissible and intersects the semiaxes at rational points. After that, this hyperplane (and the entire configuration) can be shifted by some vector with nonnegative integer coordinates in such a way that the shifted hyperplane intersects the positive semiaxes at integral points. This shift corresponds to multiplication of all operators in $\TT$ by some monomial $D^\beta$. The details of these procedures are described in \cite[pp. 3237--3238]{KMSpap}. We only mention that such rotation of $\Lambda$ is an elementary geometrical procedure and that the following spaces are isomorphic:
	$$
	C_0^\TT(\T^n) \qquad \hbox{and} \qquad C_0^{\{D^\beta T_1,\ldots, D^\beta T_J\}}(\T^n).
	$$
	The isomorphism is given by the map $f\mapsto D^\beta f$. Therefore, these operations do not change anything. Besides that, using this multiplication by $D^\beta$ we can ensure that the senior part of any operator from $\TT$ does not contain monomials of the form $\partial_j^{m_j}$.
	
	From now on we assume that the hyperplane $\Lambda$ intersects positive semiaxes at integral points and hence its equation can be written in the form $x\cdot a = k$, where $a=(a_1, a_2, \ldots, a_n)$ is a fixed vector with positive integer coordinates and $k$ is a natural number. Without loss of generality we also assume that $a_1\geq a_2\geq\ldots \geq a_n>0$.
	
	\subsection{Modification of the collection $\TT$}
	
	Now we write any operator $T_j \in \TT$ in the form $T_j=\sum_{a\cdot \alpha \leq k} c_{\alpha j}D^\alpha$. Then the senior part can be written as follows: $\sigma_j = \sum_{a\cdot \alpha = k} c_{\alpha j}D^\alpha$ and the junior part equals $\tau_j=T_j-\sigma_j=\sum_{a\cdot \alpha < k} c_{\alpha j}D^\alpha$.
	
	We denote by $\Pi_j$ the characteristic polynomial of $\sigma_j$. To be more precise, $\Pi_j$ is the following polynomial:
	$$
	\Pi_j(\xi_1, \ldots, \xi_n) = \sum_{a\cdot \alpha = k} c_{\alpha j}(2\pi i \xi_1)^{\alpha_1}\ldots (2\pi i \xi_n)^{\alpha_n},
	$$
	that is, $\Pi_j$ is a polynomial on $\R^n$ whose restriction to $\mathbb{Z}^n$ is $\widehat{\sigma_j}$. Clearly, since there are at least two linearly independent operators among $\sigma_j$, at least two of these polynomials are linearly independent, too. However, we will need the following stronger statement from \cite{KMSpap}.
	
	\begin{fact}
		Without loss of generality, we may assume that the following two polynomials of two variables are linearly independent:
		$$
		\Pi_1(\xi_1^{a_1}, \xi_2^{a_2}, \xi_2^{a_3},\ldots, \xi_2^{a_n}) \quad \hbox{and} \quad \Pi_2(\xi_1^{a_1}, \xi_2^{a_2}, \xi_2^{a_3},\ldots, \xi_2^{a_n}).
		$$
	\end{fact}
	This statement is clear for $n=2$ since in this case we only have to renumber the polynomials and operators. For $n>2$ the proof may be found in \cite{KMSpap}, see Lemma 2.3 there. It is rather elementary but quite bulky. This proof requires the procedures of passing to a complemented subspace of $C^{\TT}(\T^n)$ and modification of the collection $\TT$ without changing its linear span. Clearly, these procedures are innocent in our context.
	
	Now we have the following linearly independent polynomials: $A(\xi_1, \xi_2)=\Pi_1(\xi_1, \xi_2^{a_2}, \xi_2^{a_3},\ldots, \xi_2^{a_n})$ and $B(\xi_1, \xi_2)=\Pi_2(\xi_1, \xi_2^{a_2}, \xi_2^{a_3},\ldots, \xi_2^{a_n})$. Both $A$ and $B$ consist of linear combination of monomials of the form $\xi_1^u \xi_2^v$ where $a_1u+v=k$. Without loss of generality, such a monomial with the greatest exponent $u$ is contained in $A$ (this means that $A$ contains the monomial $\xi_1^{u}\xi_2^{k-a_1 u}$ and for $x>u$ neither $A$ nor $B$ contains $\xi_1^x \xi_2^{k-a_1 x}$). Subtracting a multiple of $T_1$ from $T_2$, we may assume that $B$ does not contain the monomial $\xi_1^{u}\xi_2^{k-a_1u}$. Besides that, multiplying $T_1$ by a constant, we make the coefficient of $\xi_1^{u}\xi_2^{k-a_1u}$ equal to $-1$. Next, if $u_1<u$ is the greatest exponent $y$ such that the monomial $\xi_1^y \xi_2^{k-a_1 y}$ is contained in $B$, we also assume that it occurs in $B$ with coefficient $1$. Also, as above, we expel this monomial from $A$.
	
	\subsection{A manipulation with Fourier coefficients}
	
	We will need yet another technical procedure, this time not contained in \cite{KMSpap}. At this moment, our motivation might not be quite clear but we will use the results of this subsection in the future computations.
	
	For a real number $x$ we denote by $[x]$ its integral part, i.e., the greatest integer not exceeding $x$. For any collection of residues $r_3, r_4, \ldots, r_n$ modulo $1000$ consider the set $A_{r_3, \ldots, r_n}=\{q\in \mathbb{N}: q \divby 1000;\ [q^{a_3/a_2}]\equiv r_3 \  (\mathrm{mod}\  1000)\ldots, \ [q^{a_n/a_2}]\equiv r_n \  (\mathrm{mod}\  1000)\}$. The (finite) union of all these sets is obviously the set of all positive integers that are divisible by $1000$, so for at least one collection of residues $r_3, \ldots, r_n$ we have:
	$$
	\sum_{q\in A_{r_3, \ldots, r_n}} q^{-1}=\infty.
	$$
	We fix these residues $r_3, r_4, \ldots, r_n$ and denote the set $A_{r_3, \ldots, r_n}$ by $\mathcal{A}$. Next, for any Banach space $X$ of functions on $\T^n$ we denote by $\widetilde{X}$ its subspace that consists of functions $f$ on $\T^n$ having non-zero Fourier coefficients $\hat{f}(p, q, s_3, \ldots, s_n)$ only if $p$ and $q$ are divisible by $1000$ and $s_j\equiv r_j \  (\mathrm{mod}\  1000)$ for all $j\geq 3$. 
	For the sake of convenience, we put $r_1$ and $r_2$ equal to $0$. We see that the subspace $\widetilde{C}_0^{\TT}(\T^n)$ is complemented in $C_0^\TT(\T^n)$; a projection is given by convolution with an appropriate measure (that is, let $\mu$ be the measure on $\T$ such that $\widehat{\mu}=\mathbf{1}_{1000\mathbb{Z}}$ and consider the measure $\nu$ defined on $\T^n$ by $\nu=\mu\otimes\mu\otimes z_3^{r_3}\mu\otimes\ldots\otimes z_n^{r_n}\mu$). Therefore, it suffices to prove Theorem 1 for this space.
	
	\section{The main argument}
	
	\subsection{The main ideas and an outline of the proof}
	
	In this subsection we outline the proof of Theorem 1, postponing certain technical details until subsequent subsections. 
	
	Recall that the hyperplane $\Lambda=\{\alpha: \alpha\cdot a = k\}$ intersects the positive semiaxes at integral points and so the numbers $m_j= k/a_j$ are integers. Consider the following collection of differential monomials: $\HH=\{D^\alpha: \alpha\in \Lambda\}$ and the corresponding space of smooth functions $C_0^\HH(\T^n)$.
	
	\begin{fact}
		We have a continuous embedding $j: \widetilde{C}_0^\HH(\T^n)\to \widetilde{C}_0^\TT(\T^n)$.
	\end{fact}
\begin{proof}
	Note that
	$$
	\|p\|_{\widetilde{C}_0^\HH(\T^n)}=\sum_{a\cdot\beta=k}\|D^\beta p\|_{C(\T^n)}\geq \sum_{l=1}^n \|\partial_l^{m_l}p\|_{C(\T^n)}, 
	$$
	and, by \cite[Theorem 9.5]{BesIlNik} the following inequality holds:
	$$
	\max_{a\cdot\beta<k}\|D^\beta p\|_{C(\T^n)}\lesssim \sum_{l=1}^n \|\partial_l^{m_l} p\|_{C(\T^n)}. 
	$$
	Hence we have:
	$$
	\|p\|_{\widetilde{C}_0^\TT(\T^n)}\lesssim \sum_{a\cdot\beta=k}\|D^\beta p\|_{C(\T^n)}+\sum_{a\cdot\beta < k}\|D^\beta p\|_{C(\T^n)}\lesssim 	\|p\|_{\widetilde{C}_0^\HH(\T^n)}.
	$$
\end{proof}	
	%Clearly, in order to prove the continuity of this embedding, one must only verify the following inequality for trigonometric polynomials $p$: 
	%$$
	%\|D^\beta p\|_{C(\T^n)} \lesssim \sum_{l=1}^n \|\partial_l^{m_l} p\|_{C(\T^n)}\lesssim \|p\|_{\widetilde{C}_0^\HH(\T^n)}, 
	%$$
	%whenever $\sum_{l=1}^n\frac{\beta_l}{m_l}<1, \ \text{i.e.} \ a\cdot\beta<k$. This fact may be found for example in the book \cite{BesIlNik}, see Theorem 9.5 there, which comes from the integral representations of $f$ and $D^\beta f$.
	
	Next, we simply consider the canonical map $i:\widetilde{C}_0^\TT(\T^n)\to \widetilde{C}_0^{T_1, T_2}(\T^n)$ (the latter space here is defined in a clear way). Recall that the senior parts of the operators $T_1$ and $T_2$ are linearly independent.
	
	After that, we have one more embedding $g: \widetilde{C}_0^{T_1, T_2}(\T^n)\to \widetilde{W}_1^{T_1, T_2}(\T^n)$. We see that $g$ is a $1$-summing operator (this follows easily from Pietsch factorization theorem, see \cite[p.203]{Wojt}; the operator $g$ is a part of the canonical embedding of $C(\T^n)\oplus C(\T^n)$ into $L^1(\T^n)\oplus L^1(\T^n)$).
	
	The next step of the proof is to construct a more sophisticated operator $s$ from the space $\widetilde{W}_1^{T_1, T_2}(\T^n)$ into $L^2(\mathbb{R}^2; |\xi_2|^{\frac{a_1}{a_2}-1})$. This construction may be found in Subsection 3.2. The main ingredient of the construction is Fact 1. In this way, we have the following diagram:
	$$
	\widetilde{C}_0^\HH(\T^n)\xrightarrow{j}\widetilde{C}_0^\TT(\T^n)\xrightarrow{i}\widetilde{C}_0^{T_1, T_2}(\T^n)\xrightarrow{g} \widetilde{W}_1^{T_1, T_2}(\T^n)\xrightarrow{s} L^2(\R^2; |\xi_2|^{\frac{a_1}{a_2}-1}).
	$$
	
	At this moment, we should apply some facts from the Banach space theory. The subsequent procedure is essentially the same as in the paper \cite{KisSid}.
	
	Suppose that the space $\widetilde{C}_0^\TT(\T^n)$ has local unconditional structure. Then we apply the following fact (for the proof, see \cite{GorLew} or \cite[Sec. 23]{Pietsch}).
	\begin{fact}
		Let $X$ be a Banach space with local unconditional structure. Then every $1$-summing operator $T$ from $X$ to an arbitrary Banach space $Y$ can be factorized through the space $L^1$, i.e., there is a measure $\mu$ and operators $V:X\to L^1(\mu)$ and $U: L^1(\mu)\to Y^{**}$ such that $UV = \kappa T$, where $\kappa:Y\to Y^{**}$ is the canonical embedding. 
	\end{fact}
	
	It gives us the following diagram:
	
	\begin{tikzcd}
	\widetilde{C}_0^\HH(\T^n) \arrow[r, "j"] & \widetilde{C}_0^\TT(\T^n) \arrow[r, "i"] \arrow[rd, "V"] & \widetilde{C}_0^{T_1, T_2}(\T^n) \arrow[r, "g"]  & \widetilde{W}_1^{T_1, T_2}(\T^n) \arrow[d, "s"] \\
	&                                  & L^1(\mu) \arrow[r, "U"] & L^2(\R^2; |\xi_2|^{\frac{a_1}{a_2}-1})              
	\end{tikzcd}
	
	After that we consider the dual diagram.
	
	%\begin{tikzcd}
	%\widetilde{C}_0^\HH(\T^n)^* & \widetilde{C}_0^\TT(\T^n)^* \arrow[l, "j^*"'] & \widetilde{C}_0^{T_1, T_2}(\T^n)^* \arrow[l, "i^*"']    & \widetilde{W}_1^{T_1, T_2}(\T^n)^* \arrow[l, "g^*"']                \\
	%&                   & L^\infty(\mu)  \arrow[lu, "V^*"'] &   L^2(\R^2; |\xi_2|^{\frac{a_1}{a_2}-1}) \arrow[u, "s^*"'] \arrow[l, "U^*"']
	%\end{tikzcd}
	
	Next, we construct an operator $P: \widetilde{C}_0^\HH (\T^n)^* \rightarrow W_{1/2}^\HH(\T^n)$. This construction is described in Subsection 3.3. Its main idea is to view the space $\widetilde{W}_2^\HH (\T^n)$ as a subspace of $L^2(\T^n)\oplus\ldots\oplus L^2(\T^n)=\bigoplus_{\alpha\in\Lambda}L^2(\T^n)$. Then $P$ is an orthogonal projection onto this subspace. It turns out that $P$ also acts from $\widetilde{C}_0^\HH (\T^n)^*$ to $W_{1/2}^\HH(\T^n)$ and it gives us the following diagram:
	
	\begin{equation}
	\label{diag_fin1}
	\begin{tikzcd}
	\widetilde{C}_0^\HH(\T^n)^* \arrow[d, "P"'] & \widetilde{C}_0^\TT(\T^n)^* \arrow[l, "j^*"'] & {\widetilde{C}_0^{T_1, T_2}(\T^n)^*} \arrow[l, "i^*"'] & {\widetilde{W}_1^{T_1, T_2}(\T^n)^*} \arrow[l, "g^*"']                     \\
	W_{1/2}^\HH(\T^n)                                             &                                               & L^\infty(\mu) \arrow[lu, "V^*"']                       & L^2(\R^2; |\xi_2|^{\frac{a_1}{a_2}-1}) \arrow[u, "s^*"'] \arrow[l, "U^*"']
	\end{tikzcd}
	\end{equation}
	
	Now, we argue as follows. We note that $L^\infty(\mu)$ is a space of type $C(K)$ (for some not necessarily metric compact $K$) and $W_{1/2}^\HH(\T^n)$ is a quasi-Banach space of cotype $2$. Then, by a generalization of Grothendieck theorem (it was proven for example in \cite{KislOp}), the operator $P j^* V^*$ is $2$-summing. Therefore, the operator $P j^* i^* g^* s^*$ is $2$-summing, too. By definition this operator should transform weakly $2$-summable sequences in $ L^2(\R^2; |\xi_2|^{\frac{a_1}{a_2}-1})$ into $2$-summable sequences in $W_{1/2}^\HH(\T^n)$. We will exhibit a counterexample, which leads to a contradiction. These computations are presented in Subsection 3.4 and after that the proof of the absence of local unconditional structure in the space $\widetilde{C}_0^\TT(\T^n)$ is finished.
	
	Let us now pass to the second assertion of Theorem 1. Suppose on the contrary that $\widetilde{C}_0^\TT (\T^n)^*$ is isomorphic to a closed subspace of a space $Y$ with local unconditional structure and that $Y$ does not contain the spaces $\ell_\infty^k$ uniformly.
	
	We denote the operator of isometric embedding of $C_0^\TT(\T^n)^*$ into $Y$ by $R$. The space $Y^*$ also has local unconditional structure. Therefore, we may apply Fact 4 to the operator $(sgi)^{**}R^*$ (it is also $1$-summing by local reflexivity principle; see, for instance, \cite[Sec. 28.1]{Pietsch}); hence $i^* g^* s^* R^{**}$ factors through $L^\infty(\mu)$ and we get the following diagram.
	
	\begin{equation}
	\label{diag_fin2}
	\begin{tikzcd}
	\widetilde{C}_0^\HH(\T^n)^* \arrow[d, "P"'] & \widetilde{C}_0^\TT(\T^n)^* \arrow[l, "j^*"'] \arrow[d, "\kappa R"'] & {\widetilde{C}_0^{T_1, T_2}(\T^n)^*} \arrow[l, "i^*"'] & {\widetilde{W}_1^{T_1, T_2}(\T^n)^*} \arrow[l, "g^*"']                   \\
	W_{1/2}^\HH(\T^n)                                         & Y^{**}                                                               & L^\infty(\mu) \arrow[l, "T"']                          & L^2(\R^2; |\xi_2|^{\frac{a_1}{a_2}-1}) \arrow[u, "s^*"'] \arrow[l, "S"']
	\end{tikzcd}
	\end{equation}
	
	Here $\kappa: Y\to Y^{**}$ is the canonical embedding. We apply the local reflexivity principle once again to conclude that the space $Y^{**}$ does not contain the spaces $\ell_\infty^k$ uniformly (since the same is true for $Y$). Now we use the following fact.
	
	\begin{fact}
		Let $Y$ be a Banach space that does not contain the spaces $\ell_\infty^k$ uniformly. Then there exists a number $p$, $2\leq p<\infty$ such that any operator $T: C(K)\rightarrow Y$ is $p$-summing \emph(and $\pi_p (T)\lesssim \|T\|$\emph).
	\end{fact}
	
	The proof of this fact may be found in the paper \cite{PisMaur}. We conclude that the operator $T$ is $p$-summing for some $p\geq 2$. Then the operator $TS$ is also $p$-summing and the operator $i^*g^*s^*$ is $p$-summing as a part of a $p$-summing operator. Therefore, $P j^* i^* g^* s^*$ is a $p$-summing operator from a Banach space to a quasi-Banach space of cotype $2$. We may conclude that it is $2$-summing (this fact may also be found in \cite{KislOp}).
	
	Therefore, if we prove that the operator $P j^* i^* g^* s^*$ is not $2$-summing, then this leads us to a contradiction and to the proof of the second assertion of Theorem 1. We again refer the reader to Subsection 3.4 where direct computations are presented and it is proved that the operator $P j^* i^* g^* s^*$ is not $2$-summing.
	
	In the next three subsections all omitted details of our proof are given.
	
	\subsection{Construction of an operator from $\widetilde{W}_1^{T_1, T_2}(\T^n)$ into a Hilbert space}
	
	Now we describe the construction of the operator $s$ from $\widetilde{W}_1^{T_1, T_2}(\T^n)$ to a Hilbert space, namely, $L^2(\mathbb{R}^2; |\xi_2|^{\frac{a_1}{a_2}-1})$. The main tool here will be the embedding theorem from \cite{KMSpap}, namely, Fact 1.
	
	For any function $f\in \widetilde{W}_1^{T_1, T_2}(\T^n)$ we consider the pair of functions $(f_1, f_2)=(T_1 f, T_2 f)$ which lies in $L^1(\T^n)\oplus L^1(\T^n)$. We can easily eliminate the terms $z_1^{k_1}\ldots z_n^{k_n}$ with $k_i\le 1000$ for some $i$ from the function $f$ (and therefore $f_1$ and $f_2$) by taking the 1001st remainder of its Fourier series with respect to each variable (this is a continuous linear operation). Now take a function $\psi \in C_0^\infty (\R^n)$ such that $0\leq \psi \leq 1$ everywhere, $\psi(\xi)=1$ when $\max_j |\xi_j|\leq 10$ and $\psi (\xi)=0$ when $\max_j |\xi_j|>20$. Next, we put $\phi=\check{\psi}$ and extend the functions $f_1$ and $f_2$ to $\R^n$ periodically (we use the same notation for these extended functions).
	
	The Fourier transforms of all derivatives of $\phi$ are supported on the cube with the center at the origin and side length $40$. Besides that, the functions $f_1$ and $f_2$ on the torus have no spectrum in the cube of side length $2000$ centered at the origin. Therefore, when we extend the functions $f_1$ and $f_2$ periodically to $\R^n$ and multiply them by the derivatives of $\phi$, we certainly get functions whose Fourier transform vanishes on the ball $\{\xi: |\xi|\leq 100\}$. We denote the space of functions in $L^1(\R^n)$ with this property by $L^1_\nul(\R^n)$. 
	
	We note that obviously the functions $f_1$ and $f_2$ satisfy the following equation: $-T_1 f_2 + T_2 f_1 = 0$. Now, consider the expression $ -T_1(\phi f_2)+T_2(\phi f_1)$. It is not difficult to see that we can regroup the terms and rewrite it in the following way:
	$$
	-T_1(\phi f_2)+T_2(\phi f_1) = \phi(-T_1 f_2 + T_2 f_1)+\sum_{a\cdot \alpha < k} b_\alpha D^\alpha \gamma_\alpha=\sum_{a\cdot \alpha < k} b_\alpha D^\alpha \gamma_\alpha.
	$$
	
	Here, $\gamma_\alpha$ is equal to some linear combination of products of functions $f_j$ and certain derivatives of $\phi$. Therefore, these functions depend linearly and continuously on $f$. For $j=1,2$ writing $T_j$ as $\sigma_j+\tau_j$, we obtain the following equation:
	$$
	-\sigma_1(\phi f_2)+\sigma_2(\phi f_1) +\hbox{junior terms}=0.
	$$
	The ``junior terms" here are the expressions of the form $D^{\alpha}\mu_\alpha$ (with $\alpha\cdot a < k$) and $\mu_\alpha$ is a linear combination of products of functions $f_j$ and certain derivatives of $\phi$.
	
	Now we are going to get rid of the junior terms precisely like this was done in \cite{KMSpap}.  We are going to use the following multiplier theorem, which is not difficult and was proven in \cite{KMSpap}.
	
	\begin{fact}
		Suppose that $u$ is a function on $\R^n\setminus \{0\}$ such that for all $t>0$ and for some positive numbers $b_1, \ldots, b_n$ and $\gamma$ we have:
		$$
		u(t^{b_1}\xi_1, t^{b_2}\xi_2 \ldots, t^{b_n}\xi_n)=t^{-\gamma}u(\xi_1, \xi_2, \ldots, \xi_n).
		$$
		Then the Fourier multiplier $M_u$ with the symbol $u$ \emph(that is, the operator taking $f$ to $\mathcal{F}^{-1}[u\mathcal{F}f]$\emph) is a bounded linear operator on $L^1_\nul(\R^n)$.
	\end{fact}
	
	Recall that we have junior terms of the form $D^\beta \eta$, where $\eta$ is a function in $L^1_\nul$ and $a\cdot \beta < k$, which means that
	$$
	\frac{\beta_1}{m_1}+\ldots+\frac{\beta_n}{m_n}<1.
	$$
	Consider the Fourier multiplier $R_j$ with the following symbol:
	\begin{multline*}
	u_j(\xi_1, \ldots, \xi_n)=\\
	\frac{(2\pi i\xi_1)^{\beta_1} \ldots (2\pi i\xi_{j-1})^{\beta_{j-1}} (2\pi i\xi_j)^{\beta_j+3m_j} (2\pi i\xi_{j+1})^{\beta_{j+1}} \ldots (2\pi i\xi_n)^{\beta_n} }{(2\pi i \xi_1)^{4m_1} +\ldots+ (2\pi i \xi_n)^{4m_n}}.
	\end{multline*}
	This operator satisfies the conditions of Fact 6, because if we make the substitution 
	$$
	\xi\mapsto (t^{\frac{1}{4m_1}} \xi_1, \ldots, t^{\frac{1}{4m_n}} \xi_n),
	$$
	then the function $u_j$ will be multiplied by
	$$
	t^{-1}t^{\frac{1}{4}(\frac{\beta_1}{m_1} +\ldots +\frac{\beta_n}{m_n} )+\frac{3}{4}}=t^{-\gamma}\quad \hbox{for some positive} \ \gamma.
	$$
	Therefore, we see that we may apply $R_j$ to $\eta$  and get a function from $L^1_\nul(\R^n)$. Next, we see that we may express $D^\beta \eta$ using these multipliers:
	$$
	D^\beta \eta = \sum_{j=1}^n \partial_j^{m_j} R_j \eta.
	$$ 	
	This formula can be verified easily by taking Fourier transforms of both parts of the equation. Now we use these multipliers to express the junior terms and arrive at an equation of the following form:
	\begin{equation}
	-\sigma_1 (\phi f_2) + \sigma_2(\phi f_1)+\sum_{j=1}^n \partial_j^{m_j}\omega_j = 0,
	\label{omega_def}
	\end{equation}
	where $\omega_j$ is a function on $\R^n$ that can be written as a linear combination of our multipliers applied to the products of $f_j$ and $\phi$ or certain derivatives of $\phi$. Recall that the operators $\sigma_1$ and $\sigma_2$ do not contain the monomials of the form $\partial_j^{m_j}$ (see Subsection 2.1). Therefore, finally, we may rewrite \eqref{omega_def} in the following form:
	\begin{equation}
	\sum_{a\cdot \alpha = k} D^\alpha \nu_{\alpha} = 0,
	\label{eq_with_nu}
	\end{equation}
	where for $\alpha = (0, 0, \ldots, m_j, 0,\ldots, 0)$ (here $m_j$ is at the $j$th position) $\nu_\alpha = \omega _j$ and for all other indices $\alpha$ the function $\nu_\alpha$ comes from the senior terms and is a linear combination of functions $\phi f_1$ and $\phi f_2$. Now we take the Fourier transform of the left hand side of this equation and restrict the resulting equation to the set of points of the form $(\xi_1, |\xi_2|, |\xi_2|^{a_3/a_2}, \ldots, |\xi_2|^{a_n/a_2})$. We get the following equation:
%	\begin{multline*}
%	\sum_{a\cdot \alpha = k} (2\pi i \xi_1)^{\alpha_1}(2\pi i)^{\alpha_2+\ldots+\alpha_n} |\xi_2|^{\frac{\alpha_2 a_2+\ldots + \alpha_n a_n}{a_2}}\\ \times \hat{\nu}_{\alpha}(\xi_1, |\xi_2|, |\xi_2|^{a_3/a_2}, \ldots, |\xi_2|^{a_n/a_2})=0.
%	\end{multline*}
%	It can be rewritten in the following form:
	$$
	\sum_{a\cdot \alpha = k} \xi_1^{\alpha_1} |\xi_2|^{\frac{k-a_1\alpha_1}{a_2}}\big[(2\pi i)^{\alpha_1+\ldots+\alpha_n} \hat{\nu}_\alpha (\xi_1, |\xi_2|, |\xi_2|^{a_3/a_2}, \ldots, |\xi_2|^{a_n/a_2})\big]=0.
	$$
	Next, we introduce the following notation:
	$$
	h_s(\xi_1, \xi_2)=\sum_{a\cdot \alpha = k, \ \alpha_1 = s} (2\pi i)^{\alpha_1+\ldots+\alpha_n} \hat{\nu}_\alpha (\xi_1, |\xi_2|, |\xi_2|^{a_3/a_2}, \ldots, |\xi_2|^{a_n/a_2}).
	$$
	We emphasize that these functions $h_s$ are the Fourier transforms of certain $L^1$-functions $\rho_s$ (which are linear combinations of functions $\nu_\alpha$ from equation \eqref{eq_with_nu} and therefore depend continuously on the initial  function $f$) restricted to the set $(\xi_1, |\xi_2|, |\xi_2|^{a_3/a_2}, \ldots, |\xi_2|^{a_n/a_2})$. Let us precise that $h_{m_1}$ is the restriction of $(2\pi i)^{m_1}\widehat{\omega}_1$, where $\omega_1$ is defined in \eqref{omega_def}. Using this notation, we rewrite our equation \eqref{eq_with_nu} in the following way:
	\begin{equation}
	\sum_{s=0}^{m_1} \xi_1^s |\xi_2|^{\frac{(m_1-s)a_1}{a_2}}h_s(\xi_1, \xi_2)=0.
	\label{solvab}
	\end{equation}
	This equation gives the compatibility condition for the solvability of system of equations in Fact 1 (see the remark below Fact 1; we apply it with $\theta_j=a_j/a_2$, $\varkappa=a_1/a_2$, $l=m_1$). 
	\begin{fact}
		Equation \eqref{solvab} and Fact 1 for these parameters imply the existence of functions $\eta_1, \ldots, \eta_{m_1}$ as defined in Fact 1.
	%	\begin{equation}
	%	\begin{cases}
	%	\xi_1 \psi_1(\xi)&=h_0(\xi);\\
	%	\xi_1 \psi_2(\xi)-|\xi_2|^{a_1/a_2}\psi_1(\xi)&=h_1(\xi);\\
	%	\qquad \vdots &= \quad \vdots;\\
	%	\xi_1\psi_j(\xi)-|\xi_2|^{a_1/a_2}\psi_{j-1}(\xi) &= h_{j-1}(\xi);\\
	%	\qquad \vdots &= \quad \vdots;\\
	%	\xi_1\psi_{m_1}(\xi)-|\xi_2|^{a_1/a_2}\psi_{m_1-1}(\xi) &= h_{m_1-1}(\xi);\\
	%	\qquad \qquad -|\xi_2|^{a_1/a_2}\psi_{m_1}(\xi)&=h_{m_1}(\xi).
		
	%	\end{cases}
	%	\label{MainSystem}
	%	\end{equation}
	\end{fact}
	We remind the reader that this simple statement (which may be proved by induction) may be found in the paper \cite{KMSpap}. Therefore, we may use it and get the following inequality:
	$$
	\max_j \Big( \int_{\R^2} |\eta_j (\xi)|^2 |\xi_2|^{\frac{a_1}{a_2}-1} \, d\xi \Big)^{1/2} \lesssim \max \|\rho_s\|_{L^1(\R^n)}\lesssim \max_{\alpha\cdot a=k}\|\nu_\alpha\|_{L^1}.
	$$
	This completes the construction of our operator into $L^2(\R^2; |\xi_2|^{\frac{a_1}{a_2}-1})$ --- it maps the function $f\in \widetilde{W}_1^{T_1, T_2}(\T^n)$ to the functions $\eta_j$ (for the moment, the reader may think that we have an operator to the direct sum of $m_1$ copies of $L^2(\R^2; |\xi_2|^{\frac{a_1}{a_2}-1})$; however, we will need only one function $\eta_j$, namely, $\eta_u$, where $u$ is exactly the number from Subsection 2.2; we will recall it later when we do the final computations). We denote the operator we have constructed by $s$. At this moment, we have the following diagram.
	$$
	\widetilde{C}_0^\HH(\T^n)\xrightarrow{j}\widetilde{C}_0^\TT(\T^n)\xrightarrow{i}\widetilde{C}_0^{T_1, T_2}(\T^n)\xrightarrow{g} \widetilde{W}_1^{T_1, T_2}(\T^n)\xrightarrow{s} L^2(\R^2; |\xi_2|^{\frac{a_1}{a_2}-1}).
	$$
	
	\subsection{Construction of an operator into the space $W_{1/2}^\HH(\T^n)$}
	
	In this subsection, we construct an operator $P$ acting on the space $\widetilde{C}_0^\HH(\T^n)^*$.
	
	Let us consider the space $\widetilde{W}_2^\HH (\T^n)$. 
		%The definition of this space should be self-explanatory; the space $W_2^\HH(\T^n)$ is determined by the seminorm $\max_{T\in \HH}\|Tf\|_{L^2(\T^n)}$ and we recall that in order to define the space $\widetilde{W}_2^\HH (\T^n)$ we consider only the functions $f$ that may have nonzero Fourier coefficient $\hat{f}(p, q, s_3, \ldots, s_n)$ only if $p$ and $q$ are divisible by $1000$ and $s_j\equiv r_j \  (\mathrm{mod}\  1000)$ (the construction of the numbers $r_j$ was described in Subsection 2.3). 
	Recall that we identify it with a subspace of $L^2(\T^n)\oplus\ldots\oplus L^2(\T^n)=\bigoplus_{\alpha\in\Lambda}L^2(\T^n)$. %This identification is given by the following map:
	%$$
	%f\rightarrow (D^\alpha f)_{\alpha\in\Lambda}.
	%$$
	We denote by $P$ the orthogonal projection from $\bigoplus_{\alpha\in\Lambda}L^2(\T^n)$ onto $\widetilde{W}_2^\HH (\T^n)$.
	
	For a multiindex $l=(l_1,\ldots, l_n)$ and a number $\alpha\in\Lambda$ we introduce the notation $Z_\alpha^l =(0, \ldots, 0, z_1^{l_1}\ldots z_n^{l_n}, 0,\ldots , 0)\in \bigoplus_{\gamma\in\Lambda}L^2(\T^n)$, where the monomial occurs at index $\alpha\in\Lambda$. We now show that it is easy to describe the action of $P$ on these elements (which obviously form an orthonormal basis in $\bigoplus_{\alpha\in\Lambda}L^2(\T^n)$).
	\begin{lemma}
		If $l_j=0$ for some $j$ or if $l_j\not\equiv r_j \  (\mathrm{mod}\  1000)$ for some $j$, then $P(\phi_l^\alpha)=0$. Otherwise, the following formula holds true:
		$$
		P(Z_\alpha^l)=\bar{\lambda}_\alpha \Big( \sum_{\beta\in\Lambda} |\lambda_\beta|^2 \Big)^{-1} \cdot (\lambda_\beta z_1^{l_1}\ldots z_n^{l_n})_{\beta\in\Lambda}.
		$$
		Here $\lambda_\beta=(2\pi il_1)^{\beta_1}\ldots (2\pi il_n)^{\beta_n}$.
	\end{lemma}
	
	We note that the first assertion of this lemma is obvious and the formula from the second assertion may be verified by simple calculations (by the way, its analog in the two-dimensional setting was proven in \cite{KisSid}).
	
	Our next goal is to realize how $P$ acts on the space $\widetilde{C}_0^\HH (\T^n)^*$. We note that the space $\widetilde{C}_0^\HH (\T^n)$ may be identified with a subspace of $\bigoplus_{\alpha\in\Lambda} C(\T^n)$. Therefore, we have (here we denote by $\mathcal{M} (\T^n)$ the space of (complex) measures with finite total variation on $\T^n$):
	$$
	\widetilde{C}_0^\HH (\T^n)^*  = \Big( \bigoplus_{\alpha\in \Lambda}  \mathcal{M} (\T^n) \Big) / \mathcal{X},
	$$
	where $\mathcal{X}$ is the annihilator of $\widetilde{C}_0^\HH(\T^n)$ in $\Big( \bigoplus_{\alpha\in\Lambda} C(\T^n) \Big)^*$. That is,
	$$
	\mathcal{X}= \Big\{ (\mu_\alpha)_{\alpha\in\Lambda}: \sum_{\alpha\in\Lambda} \int D^\alpha g \, d\bar{\mu}_\alpha = 0 \  \forall g \in \widetilde{C}_0^\HH(\T^n) \Big\}.
	$$
	
	At this point we should consider the operator $P_M$ that is the composition of $P$ and the operator of convolution with the $M$th Fej\'er kernel in all $m$ variables. We denote by $\Phi_M$ the $M$th Fej\'er kernel on the circle. Now we define the following operators:
	$$
	P_M(F) = P\big((\Phi_M \ast \mu_\alpha)_{\alpha\in \Lambda}\big), \quad F\in \widetilde{C}_0^\HH(\T^n)^*,
	$$
	where $(\mu_\alpha)_{\alpha\in\Lambda}$ is any representative of the functional $F$. This formula is meaningful since $P$ is an \textit{orthogonal} projection and if $(\nu_\alpha)_{\alpha\in\Lambda}$ lies in $\mathcal{X}$ then $(\Phi_M \ast \nu_\alpha)_{\alpha\in \Lambda}$ lies in $\mathcal{X}\cap \bigoplus_{\alpha\in\Lambda}L^2(\T^n)$ and therefore is annihilated by $P$. Now we are going to use the following fact.
	
	\begin{fact}
		The operators $P_M: \widetilde{C}_0^\HH (\T^n)^* \rightarrow W_{1/2}^\HH(\T^n)$ are uniformly bounded in $M$.
	\end{fact}
	
	The proof of Fact 8 is based on the theory of singular integral operators (and Fourier multipliers) with mixed homogeneity developed in the paper \cite{FabRiv}. However, in \cite{FabRiv} everything was done for functions on $\R^n$ rather than on $\T^n$. From the lemma formulated above we see that the components of the operator $P$ are Fourier multipliers and we may consider the same multiplier operators for the Euclidean space $\R^n$, i.e., the Fourier multiplier with the following symbol:
	$$
	\frac{\bar{\lambda_\alpha} \lambda_\beta}{\sum_{\beta\in\Lambda}|\lambda_\beta|^2}, \quad \text{where}\ \lambda_\beta=(2\pi il_1)^{\beta_1}\ldots (2\pi il_n)^{\beta_n}.
	$$
	Here $l=(l_1,\ldots, l_n)$ is an element of $\R^n$. We note that these functions possess a certain (mixed) homogeneity: if we multiply $l_j$ by $c^{a_j}$ (where $c\in \R_+$) then the symbol of the multiplier does not change (since $\sum a_j\beta_j=k$ for any $\beta\in\Lambda$). Therefore, it follows from the results of \cite{FabRiv} (see H\" ormander--Mikhlin multiplier theorem with mixed homogeneity there) that the corresponding Fourier multipliers are uniformly of weak type $(1,1)$. Hence, in order to prove Fact 8 we simply use the transference principle which is obtained in the paper \cite{KisSid} for this context and conclude that the operators $P_M$ are uniformly of weak type $(1,1)$. 
	
	Once Fact 8 is established, we get the diagrams \eqref{diag_fin1} and \eqref{diag_fin2} (since the estimates are uniform in $M$, we omit it in our notation; in further computations, we are going to apply $P$ only to very ``nice" functions and not to measures). 
	
	\subsection{The final computations and a contradiction}
	
	In this subsection, we prove that the operator $P j^* i^* g^* s^*$ is not $2$-summing.
	
	\subsubsection{Definition of $v_{pq}$ and description of $sgij(v_{pq})$}
	
	First of all, we consider the following function: 
	$$
	v_{pq}=z_1^p z_2^q z_3^{\big[q^{\frac{a_3}{a_2}}\big]}\ldots z_n^{\big[q^{\frac{a_n}{a_2}}\big]}\in \widetilde{C}_0^\HH(\T^n).
	$$
	We recall from Subsection 2.3 that $v_{pq}\in\widetilde{C}_0^\HH(\T^n)$ means that $p$ and $q$ are divisible by $1000$ (and $q$ belongs to the set $\mathcal{A}$). Besides that, we are going to consider only the values of $q$ exceeding some big constant $C$ (which we choose later) and the values of $p$ satisfying the following inequalities:
	\begin{equation}
	\frac{\Delta}{2}q^{1/a_2}\leq p^{1/a_1}\leq \Delta q^{1/a_2}, 
	\label{key_cond}
	\end{equation}
	where $\Delta$ is also some big number to be chosen later. It will be easier for us to work with elements of unit norm in $C_0^\HH(\T^n)$, therefore we evaluate $\|v_{pq}\|_{C_0^\HH(\T^n)}$. For any operator $D^\alpha$ we see that 
	$$
	\|D^\alpha v_{pq}\|_{C(\T^n)}\asymp p^{\alpha_1}q^{\alpha_2}\big[ q^{\frac{a_3}{a_2}}\big]^{\alpha_3}\ldots \big[q^{\frac{a_n}{a_2}}\big]^{\alpha_n}. 
	$$
	We take only sufficiently large values of $q$ so that the expression $\big[q^{\frac{a_n}{a_2}}\big]$ is at least half of the quantity $q^{\frac{a_n}{a_2}}$. In this case, we may estimate our norm in the following way:
	\begin{equation}
	\|D^\alpha v_{pq}\|_{C(\T^n)}\asymp p^{\alpha_1}q^{\alpha_2} q^{\frac{a_3}{a_2}\alpha_3}\ldots q^{\frac{a_n}{a_2}\alpha_n}=p^{\alpha_1}q^{\frac{a\cdot\alpha}{a_2}} q^{-\frac{a_1\alpha_1}{a_2}}.
	\label{est_of_norm}
	\end{equation}

	The conditions on $p$ and $q$ that we imposed imply that $p^{\alpha_1}q^{-\frac{a_1\alpha_1}{a_2}}\asymp 1$ and therefore, recalling that $D^\alpha\in\HH$ if $a\cdot\alpha=k$, we have, by \eqref{est_of_norm}, $\|v_{pq}\|_{C_0^\HH(\T^n)}\asymp q^{\frac{k}{a_2}}=q^{m_2}$. We now consider $ij(v_{pq})\in \widetilde{C}_0^{T_1, T_2}(\T^n)$, noting that $(T_1 v_{pq}, T_2 v_{pq})=(c_{pq}'v_{pq}, d_{pq}'v_{pq})$. When we apply the senior part $\sigma_j$ of $T_j$ to $v_{pq}$, we get a linear combination of functions whose norms in $C(\T^n)$ are equal, by \eqref{key_cond} and \eqref{est_of_norm}, to
	$$
	p^{\alpha_1}q^{\frac{a\cdot\alpha}{a_2}} q^{-\frac{a_1\alpha_1}{a_2}}\asymp \Delta^{\alpha_1 a_1}q^{\frac{k}{a_2}}.
	$$
	We take big values of $\Delta$ and ensure that the summand with the largest $\alpha_1$ ``dominates" the others. When we apply some $D^\alpha$ from the junior part of the operator $T_j$ (which means that $a\cdot \alpha$ is less than $k$), we get a function with norm in $C(\T^n)$ equal, by \eqref{est_of_norm}, to
	$$
	p^{\alpha_1}q^{\frac{a\cdot\alpha}{a_2}} q^{-\frac{a_1\alpha_1}{a_2}}\asymp q^{\frac{a\cdot\alpha}{a_2}}=o(q^{\frac{k}{a_2}}).
	$$
	These calculations show that for big values of $C$ and $\Delta$ one summand from the expression $T_j v_{pq}$ dominates all the others and we get: $\|ij(v_{pq})\|_{C^{T_1,T_2}(\T^n)}=\max\{|c_{pq}'|, |d_{pq}'|\}\asymp q^{\frac{k}{a_2}}$. Hence, we write: $(T_1 v_{pq}, T_2 v_{pq})=q^{\frac{k}{a_2}}(c_{pq} v_{pq}, d_{pq}v_{pq})$, where $|c_{pq}|, |d_{pq}|\asymp 1$.
	
	Obviously,  $\|gij(v_{pq})\|_{W_1^{T_1,T_2}(\T^n)}$ is also equal to $\max\{|c_{pq}'|, |d_{pq}'|\}\asymp q^{\frac{k}{a_2}}$. We now consider $sgij(v_{pq})$. We apply the procedure described in Subsection 3.2 to $(f_1, f_2)=q^{-\frac{k}{a_2}}(T_1 v_{pq}, T_2 v_{pq})=(c_{pq}v_{pq}, d_{pq}v_{pq})$. We keep the notation $\psi$, $\phi$, $u_j$ and recall equations \eqref{omega_def} and \eqref{eq_with_nu}.
	
	%These functions satisfy the following equation:
	%\begin{equation}
	%-d_{pq}T_1 v_{pq} + c_{pq} T_2 v_{pq} = 0.
	%\label{eq}
	%\end{equation}
	
	%Now we take the function $\psi$ as in the Subsection 3.2. It is supported on the set $\{\xi\in\R^n: \max_j |\xi_j|\leq 20\}$ and equals $1$ on the set $\{\xi\in\R^n: \max_j |\xi_j|\leq 10\}$. We denote by $\phi$ its inverse Fourier transform. Recall that equation \eqref{eq} implies the following:
	%\begin{equation}
	%-\sigma_1 (\phi f_2) + \sigma_2(\phi f_1)+\sum_{j=1}^n \partial_j^{m_j}\omega_j = 0.
	%\label{MainEq}
	%\end{equation}
	%The functions $\omega_j$ here are obtained by applying the Fourier multipliers $R_j$ to the functions $c_{pq} \phi v_{pq}$ and $d_{pq} \phi v_{pq}$. After regrouping the terms, we rewrite this equation in the following form:
	%\begin{equation*}
	%\sum_{a\cdot \alpha = k} D^\alpha \nu_{\alpha} = 0.
	%\label{TransformedEq}
	%\end{equation*}

	The next step is to take the Fourier transform of the left-hand side of equation \eqref{omega_def}. Clearly, the Fourier transform of the function $\phi v_{pq}$ (here we assume that $v_{pq}$ is extended periodically to $\R^n$) is the following function:
	$$
	(\xi_1, \ldots, \xi_n) \mapsto \psi(\xi_1 - p, \xi_2 - q, \xi_3 - \big[ q^{\frac{a_3}{a_2}} \big], \ldots, \xi_n - \big[ q^{\frac{a_n}{a_2}} \big]).
	$$
	We restrict these functions to the set of points of the form $$(\xi_1, |\xi_2|, |\xi_2|^{a_3/a_2}, \ldots, |\xi_2|^{a_n/a_2})$$ and denote the result by $\psi_{pq}$:
	$$
	\psi_{pq}(\xi_1, \xi_2)=\psi(\xi_1 - p, |\xi_2| - q, |\xi_2|^{\frac{a_3}{a_2}}-\big[ q^{\frac{a_3}{a_2}} \big], \ldots, |\xi_2|^{\frac{a_n}{a_2}}-\big[ q^{\frac{a_n}{a_2}} \big]).
	$$
	Since $a_j \leq a_2$ for $j\geq 3$, we see that $\big||\xi_2|^{a_j/a_2}-\big[ q^{a_j/a_2} \big]\big|\leq 1+ \big||\xi_2|^{a_j/a_2}- q^{a_j/a_2} \big| \leq 1 + ||\xi_2|-q|$. Therefore, $\psi_{pq}(\xi_1, \xi_2)=1$ if $|\xi_1 - p|\leq 5$ and $||\xi_2|-q|\leq 5$. On the other hand, if $\max \{|\xi_1-p|, ||\xi_2|-q|\}>20$, then $\psi_{pq}(\xi_1, \xi_2)=0$. In particular, we conclude that (again, for big values of $p$ and $q$) on the support of the function $\psi_{pq}$ we have (recall inequalities \eqref{key_cond}):
	\begin{equation}
		\xi_1 \asymp p \asymp q^{a_1/a_2}\asymp |\xi_2|^{a_1/a_2}\asymp |\xi_2|^{m_2/m_1}
		\label{est_on_supp}.
	\end{equation}
	
	Therefore, we have the following formula for the symbols of the multipliers $R_j$ on the support of the Fourier transform of the function $\psi_{pq}$: %(we use the fact that $\xi_1 \asymp p \asymp q^{a_1/a_2}\asymp |\xi_2|^{a_1/a_2}$):
	\begin{multline}
	|u_j(\xi_1, \xi_2, |\xi_2|^{\frac{a_3}{a_2}}, \ldots, |\xi_2|^{\frac{a_n}{a_2}})|\asymp \frac{|\xi_1|^{\beta_1} |\xi_2|^{m_2 \big( \frac{\beta_2}{m_2}+\ldots+\frac{\beta_n}{m_n} \big)+3m_2}}{|\xi_1|^{4m_1}+|\xi_2|^{4m_2}}\\ \asymp \frac{|\xi_2|^{m_2 \big( \frac{\beta_1}{m_1}+\ldots+\frac{\beta_n}{m_n} \big)+3m_2}}{|\xi_2|^{4m_2}}
	\asymp |\xi_2|^{m_2 \big( -1 + \frac{\beta_1}{m_1}+\ldots+\frac{\beta_n}{m_n} \big)}\asymp |\xi_1|^{-\varepsilon},
	\label{est_on_supp_2}
	% \text{where}\ \sigma=1-\big( \frac{\beta_1}{m_1}+\ldots+\frac{\beta_n}{m_n} \ \big) > 0.
	\end{multline}
	where $\varepsilon=m_1\big(1-\big( \frac{\beta_1}{m_1}+\ldots+\frac{\beta_n}{m_n} \ \big)\big) > 0$.
	
	We introduce the notation $w_{pq}=q^{-m_2} v_{pq}$ so that $\|w_{pq}\|_{C_0^\HH(\T^n)}\asymp 1$. Our next goal is to compute the function $sgij(w_{pq})$.
	
	We arrive at the structure of the functions $h_s$ from the system of equations in Fact 1 as precised in Subsection 3.2. For the reader's convenience, we rewrite the formula defining the function $h_s$ here:
	$$
	h_s(\xi_1, \xi_2)=\sum_{a\cdot \alpha = k, \ \alpha_1 = s} (2\pi i)^{\alpha_1+\ldots+\alpha_n} \hat{\nu}_\alpha (\xi_1, |\xi_2|, |\xi_2|^{a_3/a_2}, \ldots, |\xi_2|^{a_n/a_2}).
	$$
	We see that the function $h_{m_1}$ is obtained from $\omega_1$, see the equation \eqref{omega_def} for the definition of $\omega_j$ (recall that for $\alpha = (0, \ldots, 0, m_j, \ldots, 0)$ we have  $\nu_\alpha = \omega _j$, see \eqref{eq_with_nu}) and hence $|h_{m_1}(\xi_1, \xi_2)|\lesssim |\xi_1|^{-\varepsilon}$ by \eqref{est_on_supp_2} (and of course $h_{m_1}(\xi) = 0$ when $\xi$ is not in the support of $\psi_{pq}$; we recall that some derivatives of function $\phi$ may appear in the definition of $\omega_j$, however, they do not affect our estimates). Furthermore, recall that we do not have the differential monomials $D^\alpha$ in the operators $\sigma_1$ and $\sigma_2$ if $\alpha_1$ is greater than $u$ (see Subsection 2.2). Hence, the functions $h_{m_1-1}, \ldots, h_{u+1}$ are equal to $0$. 
	
	Besides that, it is not difficult to see that $h_u = d_{pq}\psi_{pq}$. Indeed, the constructions in Subsection 2.2 ensured that in the equation \eqref{omega_def} the monomials $D^\alpha$ with $\alpha_1=u$ occur only in the summand $-\sigma_1(\phi f_2)=-d_{pq}\sigma_1(\phi v_{pq})$. Now recall that  the coefficient of the monomial $\xi_1^u \xi_2^{k-a_1 u}$ in  the polynomial $\Pi_1(\xi_1, \xi_2^{a_2}, \xi_2^{a_3},\ldots, \xi_2^{a_n})$ is equal to $-1$.
	
	Now we solve the following system of equations (these are simply the last $m_1-u+1$ equations from the system in Fact 1):
	\begin{equation*}
	\begin{cases}
	&-|\xi_2|^{\frac{a_1}{a_2}}\eta_{m_1}(\xi) = h_{m_1}(\xi);\\
	\xi_1 \eta_{j+1}(\xi)&- |\xi_2|^{\frac{a_1}{a_2}} \eta_j(\xi) = 0, \quad j=m_1-1, \ldots, u+1;\\
	\xi_1 \eta_{u+1}(\xi)&- |\xi_2|^{\frac{a_1}{a_2}} \eta_u(\xi) = h_u(\xi).
	\end{cases}
	\end{equation*}
	We solve these equations one by one, starting with the first. We get:
	\begin{gather*}
	\eta_{m_1}(\xi)=O(|\xi_2|^{-\frac{a_1}{a_2}} |\xi_1|^{-\varepsilon}), \eta_{m_1-1}(\xi)=O(|\xi_2|^{-\frac{a_1}{a_2}} |\xi_1|^{-\varepsilon}), \ldots, \\ 
	\eta_u(\xi) = |\xi_2|^{-\frac{a_1}{a_2}}(d_{pq}\psi_{pq}(\xi) + O(|\xi_1|^{-\varepsilon})).
	\end{gather*}
	Besides that, we again emphasize that $\eta_u(\xi)$ may be nonzero only on the support of the function $\psi_{pq}$.
	
	The function $\eta_u$ is the image of $w_{pq}$ under the action of $s$. Therefore, from now on we write $\eta^{(p, q)}$ instead of $\eta_u$ in order to emphasize the dependence of this function on $p$ and $q$.
	
	The norm of the function $\eta^{(p,q)}$ in $L^2(\R^2; |\xi_2|^{\frac{a_1}{a_2}-1})$ may be bounded from below in the following way (we use \eqref{est_on_supp} here):
	\begin{multline*}
	\|\eta^{(p, q)}\|_{L^2(\R^2; |\xi_2|^{\frac{a_1}{a_2}-1})}=\Big( \int_{\R^2} |\eta^{(p,q)} (\xi)|^2 |\xi_2|^{\frac{a_1}{a_2}-1} d\xi \Big)^{1/2}\gtrsim\\
	 \Big( \int_{\substack{|\xi_1-p|\leq 5 \\ ||\xi_2|-q|\leq 5}} |\xi_2|^{-2\frac{a_1}{a_2}} |\xi_2|^{\frac{a_1}{a_2}-1} d\xi \Big)^{1/2}
	\asymp \Big( \int_{\substack{|\xi_1-p|\leq 5 \\ ||\xi_2|-q|\leq 5}} |\xi_1|^{-1} |\xi_2|^{-1} d\xi \Big)^{1/2}\\
	\asymp (p^{-1}q^{-1})^{1/2}.
	\end{multline*}
	In a similar way we get the following bound for this norm from above:
	$$
	\|\eta^{(p, q)}\|_{L^2(\R^2; |\xi_2|^{\frac{a_1}{a_2}-1})} \lesssim \Big( \int_{\substack{|\xi_1-p|\leq 20 \\ ||\xi_2|-q|\leq 20}} |\xi_1|^{-1} |\xi_2|^{-1} d\xi \Big)^{1/2}\asymp (p^{-1}q^{-1})^{1/2}.
	$$
	
	Summing up, we see that the image of the function $w_{pq}$ under the action of $sgij$ is the function $\eta^{(p,q)}$ such that $\|\eta^{(p,q)}\|_{L^2(\R^2; |\xi_2|^{\frac{a_1}{a_2}-1})}\asymp (p^{-1}q^{-1})^{1/2}$. Now we take the functions $ (pq)^{1/2} \eta^{(p,q)}$ (so that the norms of these functions are bounded away from zero and infinity). Since the supports of these functions do not intersect, they are orthogonal and hence they form a weakly $2$-summable sequence in ${L^2(\R^2; |\xi_2|^{\frac{a_1}{a_2}-1})}$. 
	
	\subsubsection{A weakly $\ell^2$-summable sequence whose image by $Pj^*i^*g^*s^*$ is not $\ell^2$-summable}
	
	Now we prove that the sequence of functions $(P j^* i^* g^* s^*) ((pq)^{1/2}\eta^{(p, q)})$ is not $2$-summable in $W_{1/2}^\HH(\T^n)$.
	
	First, we evaluate the element $(j^* i^* g^* s^*)((pq)^{1/2}\eta^{(p, q)})$. Consider the following monomials:
	$$
	v_{\tilde{p}, \tilde{q}, s_3, \ldots, s_n} = z_1^{\tilde{p}} z_2^{\tilde{q}} z_3^{s_3}\ldots z_n^{s_n}\in \widetilde{C}_0^\HH(\T^n).
	$$ 
	The trigonometric polynomials are dense in $\widetilde{C}_0^\HH(\T^n)$ and therefore two elements in $C_0^\HH(\T^n)^*$ coincide if their pairings with the above monomials coincide. So the element $(j^* i^* g^* s^*)((pq)^{1/2}\eta^{(p, q)})$ is determined by the following identity (which must be true for every function $v_{\tilde{p}, \tilde{q}, s_3, \ldots, s_n}$):
	\begin{equation}
	\langle v_{\tilde{p}, \tilde{q}, s_3, \ldots, s_n}, (j^* i^* g^* s^*) (\eta^{(p, q)}) \rangle= \langle (sgij)(v_{\tilde{p}, \tilde{q}, s_3, \ldots, s_n}), \eta^{(p, q)} \rangle.
	\label{adj}
	\end{equation}
	The right-hand side of this formula is understood as the scalar product in $L^2(\R^2; |\xi_2|^{\frac{a_1}{a_2}-1})$. 
	
	The function $(sgij)(v_{\tilde{p}, \tilde{q}, s_3, \ldots, s_n})$ may be evaluated in the same way as above. We do not make any claims about the norm of this function but we may still conclude that it is supported on the set 
	$$
	\{\xi=(\xi_1, \xi_2): |\xi_1-\tilde{p}|\leq 20, ||\xi_2|-\tilde{q}|\leq 20, |s_j-|\xi_2|^{a_j/a_2}|\leq 20, \ 3\leq j\leq n\}.
	$$
	The right-hand side of \eqref{adj} may be nonzero only if the intersection of supports of functions $(sgij)(v_{\tilde{p}, \tilde{q}, s_3, \ldots, s_n})$ and $\eta^{(p,q)}$ is nonempty. Therefore if $\xi=(\xi_1, \xi_2)$ lies in the supports of both functions, we have: $|\xi_1-p|\leq 20$ and $|\xi_1 - \tilde{p}| \leq 20$. Since both $p$ and $\tilde{p}$ are divisible by $1000$, we have $p=\tilde{p}$. Similarly, we see that $q=\tilde{q}$.
	
	Next, we have: 
	$$
	|s_j-|\xi_2|^{\frac{a_j}{a_2}}| \leq 20 \quad \text{and}\quad \big| \big[ q^{\frac{a_j}{a_2}}\big] - |\xi_2|^{\frac{a_j}{a_2}}  \big| \leq 20.
	$$
	These inequalities imply that $|s_j-\big[ q^{\frac{a_j}{a_2}} \big]|\leq 40$. But since the function $v_{\tilde{p}, \tilde{q}, s_3, \ldots, s_n}$ lies in $\widetilde{C}_0^\HH(\T^n)$, we have $s_j\equiv \big[ q^{\frac{a_j}{a_2}} \big] \ (\mathrm{mod}\ 1000)$. Therefore, $s_j= \big[ q^{\frac{a_j}{a_2}} \big]$.
	
	Thus we see that the right-hand side of \eqref{adj} is nonzero only for the function 
	$$
	v_{p,q,[q^{\frac{a_3}{a_2}}], \ldots, [q^\frac{a_n}{a_2}]} = v_{pq}.
	$$
	But for this function we have:
	\begin{multline*}
	\langle (sgij)v_{pq}, \eta^{(p,q)} \rangle = 
	\langle q^{\frac{k}{a_2}} (sgij)w_{pq}, \eta^{(p,q)}  \rangle \\
	= \langle  q^{m_2} \eta^{(p,q)}, \eta^{(p,q)} \rangle = \varkappa_{pq} (pq)^{-1} q^{m_2},
	\end{multline*}
	where $|\varkappa_{pq}|\asymp 1$.
	
	Thus, we have proved that if we apply the element $(j^* i^* g^* s^*)(\eta^{(p,q)})\in \widetilde{C}_0^\HH(\T^n)^*$ to any monomial other than $v_{pq}$ we get zero and if we apply it to $v_{pq}$ we get $\varkappa_{pq}q^{m_2}(pq)^{-1}$. Now we must realize which collection of measures corresponds to $(j^* i^* g^* s^*)(\eta^{(p,q)})$.
	
	We see that the collection of functions $(D^\alpha v_{pq})_{\alpha\in \Lambda}\in \bigoplus_{\alpha\in\Lambda} C(\T^n)$ corresponds to $v_{pq}\in \widetilde{C}_0^\HH(\T^n)$. One of these functions equals $\partial_1^{m_1} v_{pq} = (2\pi i p)^{m_1}v_{pq}$. Therefore, we may take the following collection $(\mu_\alpha)_{\alpha\in\Lambda}$ corresponding to $(j^* i^* g^* s^*)(\eta^{(p,q)})$:
	\begin{gather*}
	\mu_\alpha = q^{m_2} (2\pi i p)^{-m_1} \varkappa_{pq} (pq)^{-1} v_{pq}\ \text{for}\ \alpha=(m_1, 0,\ldots,0);\\
	\mu_\alpha = 0 \ \text{for all other} \ \alpha\text{'s}. 
	\end{gather*}
	Now we apply the projection $P$ to this element and we have the formula for it (see the Lemma from Subsection 3.3). Clearly, in our case all the numbers $\lambda_\beta$ are equivalent (i.e., for $\beta, \gamma \in \Lambda$ we have $\lambda_\beta \asymp \lambda_\gamma$) because, by \eqref{key_cond}, we have:
	$$
	|\lambda_\beta| \asymp p^{\beta_1} q^{\beta_2} q^{\frac{a_3}{a_2}\beta_3}\ldots q^{\frac{a_n}{a_2}\beta_n}\asymp q^{\frac{a\cdot\beta}{a_2}}=q^{m_2}.
	$$
	Hence, we have, since $q^{m_2}p^{-m_1}\asymp 1$ by \eqref{est_on_supp}:
	\begin{multline*}
	\sum_{p,q} \|(P j^* i^* g^* s^*)((pq)^{1/2}\eta^{(p,q)})\|^2_{W_{1/2}^\HH(\T^n)}\\
	\asymp \sum_{p,q} \|(pq)^{-1/2}v_{pq}\|_{L^{1/2}}^2\asymp \sum_{p,q} (pq)^{-1}.
	\end{multline*}
	Recall that we consider only large $p$ and $q$ such that $\frac{\Delta}{2}q^{1/a_2}\leq p^{1/a_1}\leq \Delta q^{1/a_2}$ and $q\in\mathcal{A}$ (and $p$ is divisible by $1000$). For any fixed $q$ there are about $q^{a_1/a_2}$ such numbers $p$ and each of them is equivalent to $q^{a_1/a_2}$. Therefore, we can write:
	$$
	\sum_{p,q} (pq)^{-1}\asymp \sum_{q\in\mathcal{A}, q>C} q^{-1}.
	$$
	And this sum diverges (by the definition of $\mathcal{A}$) and we finally get a contradiction and the theorem is proved.
	
	\subsection*{Acknowledgements} 
	This research was supported by the Russian Science Foundation grant 19-71-30002. 	The author is kindly grateful to his scientific advisor, S. V. Kislyakov, for posing this problem, for very helpful discussions during the process of its solution and for the great help in editing of this text. Besides that, the author is grateful to the anonymous referees for their helpful recommendations concerning the presentation.

\end{document}